\documentclass[14pt]{article}

\usepackage[cp1251]{inputenc}
\usepackage[english]{babel}
\usepackage{amsfonts,amssymb,amsmath,amstext,amsbsy,amsopn,amscd,amsthm,graphicx,euscript,graphicx}
\usepackage{graphics}

\newtheorem{thm}{Theorem}
\newtheorem{lm}{Lemma}

\newtheorem{exm}{Example}

\newcounter{tdfn}
\newenvironment{dfn}
{\vspace{0.15cm}{\bf Definition \arabic{tdfn}.}} {\par
\addtocounter{tdfn}{1}}
\newcounter{trk}
\newenvironment{rk}
{\vspace{0.15cm}{\bf Remark \arabic{trk}.}} {\par
\addtocounter{trk}{1}}
{\endtrivlist}

\def\:{\colon}

\def\0{{\mathbf 0}}
\def\1{{\mathbf 1}}

\title{A Sliceness Criterion for Odd Free knots}

\author{D. Fedoseev, V. Manturov \footnote{The research was partly supported by the RFBR grant (project 16-11-10291).}}

\date{}

\begin{document}

\maketitle



\begin{abstract}
The main goal of this paper is to prove that for {\em odd free knots} --- that is free knots with all odd crossings --- the problem of sliceness (the existence of a spanning disc) has an explicit answer based on the pairing of the knot diagram chords.
\end{abstract}


The notion of {\em parity} was introduced to topology by the second named author in 2009 \cite{Sbornik}. It allowed one to strengthen many invariants of virtual knots, free knots, etc. and to create many new invariants. One of the main properties of parity is that it allows one to produce {\em picture--valued} invariants --- invariants valued in knot diagrams or their linear combinations. They make it possible to prove statements of the form {\em if a diagram is complicated enough, it reproduces itself} (see also \cite{Pic1,Pic2}).

This statement can be formulated in a compact form the following way: $$[K]=K.$$ Here on the left--hand side $K$ means a knot or a link (an equivalence class of diagrams) and on the right--hand side $K$ is a single (complicated) diagram of the same knot. The bracket $[K]$ is constructed combinatorially and this equality means that for every diagram $K'$ which is equivalent to the diagram $K$ we have $[K']=K$. By construction it means that the diagram $K$ can be obtained form the diagram $K'$ by the means of {\em smoothing} operations. 

So many properties of {\em all} diagrams of a knot (say, $K'$) are ``encoded'' in a {\em single} diagram $K$. In particular, looking at this single diagram of a (free) knot or link one can judge its non-triviality, the minimal number of crossings in the diagrams of the knot, etc.

The main result of this paper makes it possible for some (odd) diagrams to recognise another property: sliceness. One of the most important equivalence relations on classical knot is the relation of {\em concordance} --- the possibility to be connected by a cylinder in $\mathbb{R}^3\times[0,1]$ so that one knot lies in $\mathbb{R}^3\times\{0\}$ and the other lies in $\mathbb{R}^3\times\{1\}$. In classical knot theory this notion is especially important because there are different types concordance and cobordism --- topological and smooth, and they are closely connected to different problems in four-dimensional topology (see \cite{Fri,Ras}).

The notion of concordism can easily be extended to free knots. A (free) knot is called {\em slice} if it is concordant to the trivial knot.

The main result of this paper is as follows:

\begin{thm}
\label{mnthm}
If a diagram $K$ of a free knot is {\em odd}, then it is {\em slice} if and only if its chords can be {\em paired with no intersections}.
\end{thm}

Recall the main definitions from \cite{Sbornik, Cobordism}.
Here a free knot diagram is a {\em framed 4-valent graph}. In the one-component case (and we are interested in that case only) a framed 4-valent graph is either a circle (with an empty vertex set) or a 4-regular (multi)graph with the following structure: at every vertex four edges are divided into two sets of formally opposite ones. Every framed 4-valent graph has a naturally corresponding {\em chord diagram}.

Free knots are equivalence classes of framed 4-graphs with one component modulo Reidemeister moves which we will not list here since equivalent framed 4-graphs are cobordant. 
We say that two framed 4-graphs $\Gamma_{1},\Gamma_{2}$ are {\em cobordant} if there exists a {\em spanning surface} that is a 2-complex --- an image of a cylinder under a continious map in general position such that the images of its boundary are the graphs $\Gamma_{1},\Gamma_{2}$ and in the neighbourhood of the preimage of every vertex of the graph  $\Gamma_{i}$ the image of the cylinder boundary component lies in the union of the opposite halfedges.

Clearly, this relation is indeed an equivalence relation. If a framed 4-graph $\Gamma$ is equivalent to the trivial knot (a circle with no vertices) the graph is called {\em slice}. Capping the circle with a disc we obtain a {\em spanning disc} for $\Gamma$.

The first examples of non--slice framed 4-graphs were constructed in \cite{Cobordism}.

Now let us consider the last definition from Theorem \ref{mnthm}.

Let $\Gamma$ be a framed 4-graph, $D(\Gamma)$ the corresponding chord diagram. Let us call a {\em pairing} ${\cal P}$ of its chords a partition of the chords $\{d_{i}\}$ of the diagram $D(\Gamma)$ into 
sets ${\cal P}_{i}$ consisting of one or two elements each so that for every set ${\cal P}_{i}$ consisting of two chords $c_{i},d_{i}$ every endpoint of the chord $c_i$ corresponds to exactly one endpoint of the chord $d_i$. Chord diagram $C({\cal P})$ consists of the set of chords, the endpoints of which coincide with endpoints of the chords of the diagram $D$ and every chord either coincides with a chord of the diagram $D$ if the corresponding set ${\cal P}_{i}$ consists of a single element, or connects the corresponding endpoints of the two different chords of a set.

We say that {\em a pairing ${\cal P}$ has no intersections} if the chords of the diagram $C({\cal P})$ are pairwise unlinked.

From this definition it is clear that Theorem \ref{mnthm} conditions are sufficient: if we have a pairing, we can draw the diagram $D({\cal P})$ inside a disc $D^{2}$ and then glue together the points on paired chords. Inside every chord which lies in a 1-element set this glueing produces a cusp point. So we obtain a complex --- an image of a disc, all multiple points of which are double and all the double lines are segments with their endpoints on the disc boundary.

\begin{exm}
Consider Fig. \ref{lab}. The pairing is of the form $$\{\{c_{1}\},\{c_{2},d_{2}\},\{c_{3},d_{3}\}\};$$ on the right--hand side of the figure one can see a cusp point on the chord $c_1$.

\begin{figure}
\centering\includegraphics[width=300pt]{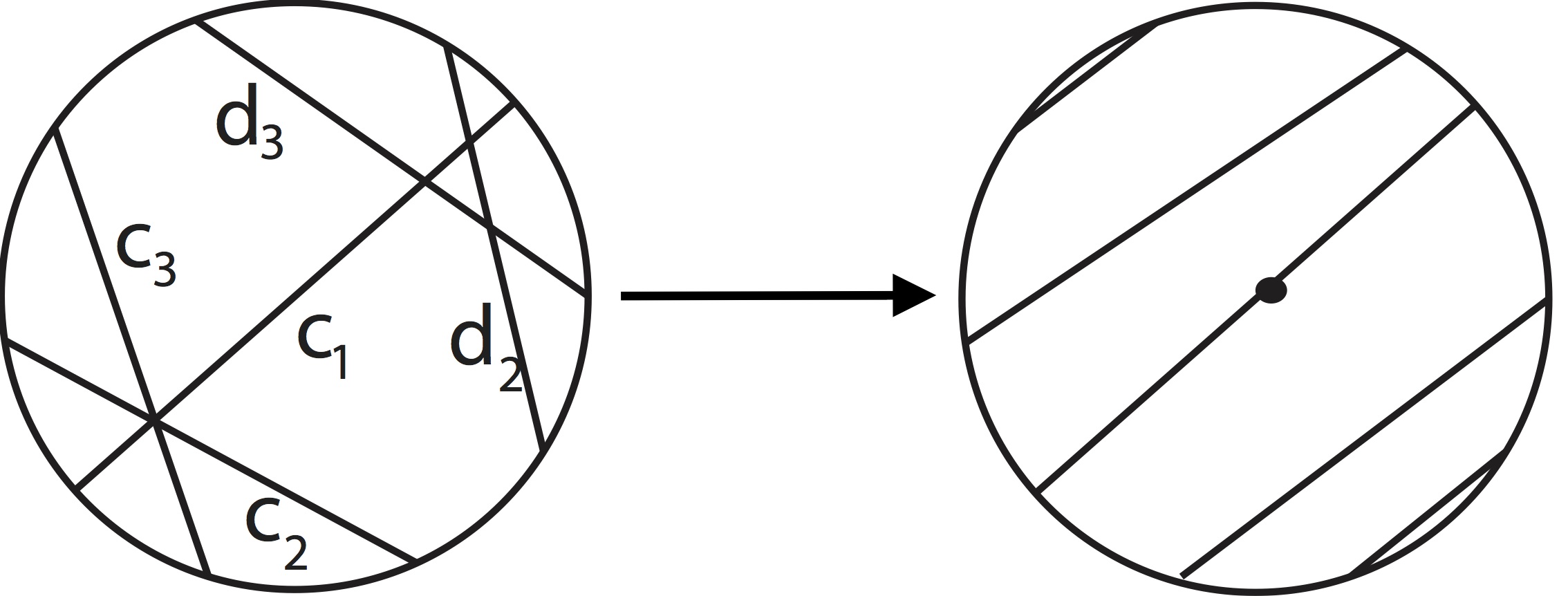}
\caption{Correct pairing yields sliceness}
\label{lab}
\end{figure} 

\end{exm}

Note that Theorem \ref{mnthm} gives a {\em finite procedure} to determine if a given odd free knot is slice.

\begin{exm}
Consider the free knot depicted on Fig. \ref{k5}.

The chord diagram on the figure has ten chords: five ``long'' ones --- $a_{1},a_{2},a_{3},a_{4},a_{5}$ and five ``short'' ones --- $b_{1},b_{2},b_{3},b_{4},b_{5}$.

Let us show that the diagram doesn't allow a pairing with no intersections. First of all, as it was mentioned before, if such a pairing exists every odd chord should be paired with another odd chord. Considering different cases it is easy to verify that two short chords can't be paired with each other. That means that every long chord should be paired to a short one.

If the chord $b_1$ is paired with the chord $a_1$, then every long chord $a_{2} ,a_{3} ,a_{4} ,a_{5}$ is paired to a long chord leading to a contradictory pairing of the short chords.

Pairing of the chord $a_1$ with the chord $b_3$ (resp. $b_4$) makes it impossible to pair the chord $b_5$ (resp. $b_2$).

Pairing of the chord $a_1$ with the chord $b_2$ (resp. $b_5$) makes it impossible to pair the chord $b_4$ (resp. $b_3$).

\begin{figure}
\centering\includegraphics[width=200pt]{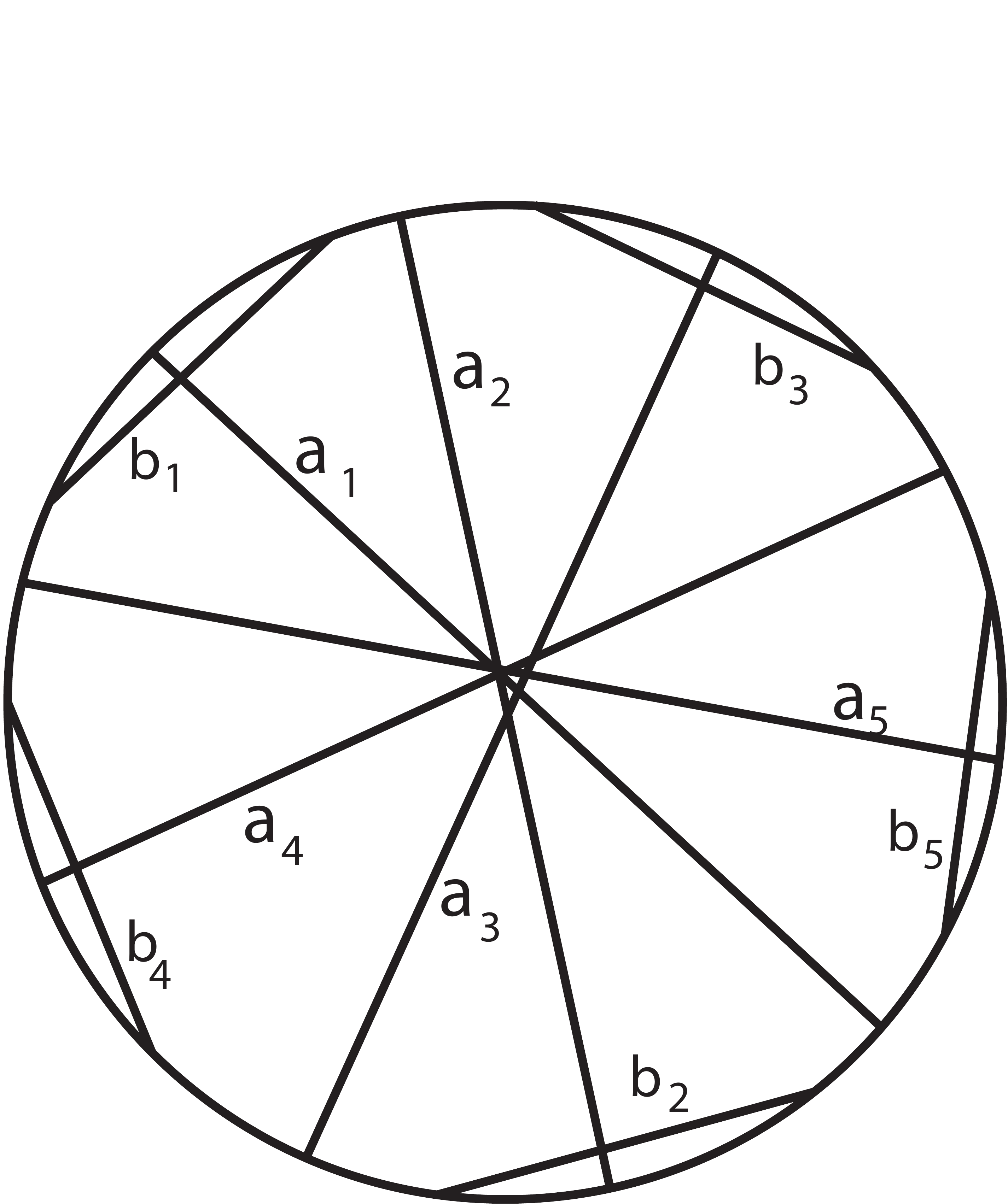}
\caption{Non--slice free knot}
\label{k5}
\end{figure}

\end{exm}

\begin{rk}
The example on Fig. \ref{lab} shows a case when a chord is paired with itself (producing a cusp) is necessary to create a spanning disc. 

It is easy to construct an example when triple points a needed. Such a diagram is shown, for example, on Fig. \ref{diag}.

\begin{figure}
\centering\includegraphics[width=300pt]{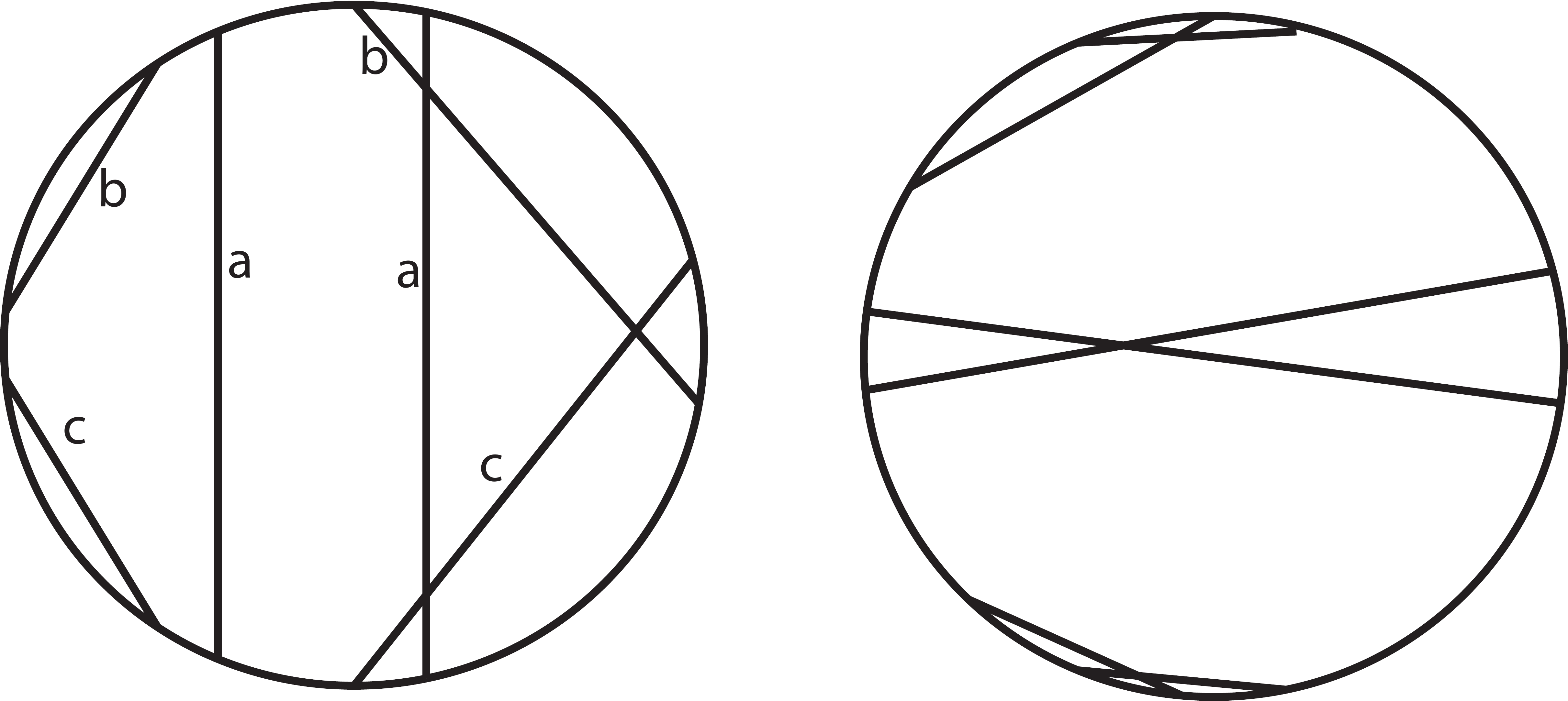}
\caption{Spanning disc with triple points}
\label{diag}
\end{figure}

On the figure on the left, the paired chords are marked with the same letters ($a$, $b$, $c$). The pairing depicted on the figure on the right has three crossings, but the corresponding double lines can be paired in such a way that the triple point appears on the crossing of the three double lines and we get a spanning disc for the knot.

The reason is that the chord diagram on the left is constructed from the diagram of the connected sum of two trivial free knots $K{\#}{\bar K}$ with an obvious pairing by applying the third Reidemeister move, and Reidemeister moves do not affect sliceness.

In general, connected sum of a free knot with its mirror image at a suitable point $K \to K{\#} {\bar K}$ allows one to construct an infinite series of non--trivial. slice free knots. \\
\end{rk}

Let $\Gamma$ be a framed 4-graph, let $K$ be a spanning complex which is an image of a 2-disc in general position. Our goal is to construct a new complex $\tilde{K}$ --- an image of a disc such that every double line starts and ends on its boundary.

The main idea behind the proof of Theorem \ref{mnthm} is the construction of some analog of a ``bracket'' (a smoothing) for a 2-complex (see \cite{Sbornik}). We construct a new complex which is the same as the previous one on its boundary but doesn't have any even double lines.

To achieve this goal we need some basic notions of the theory of 2-knots. Let us recall them following \cite{Parity, winter}.

A {\em 2-knot} (resp. {\em an $n$-component 2-link}) is a smooth embedding in general position of a 2-sphere $S^2$ (resp. disjoint union of $n$ spheres) into $\mathbb{R}^4$ or $S^4$ up to isotopy.

As in case of 1-dimensional knots, 2-knots can be presented with {\em diagrams}. 

A {\em diagram} of a 2-knot $K$ in $\mathbb{R}^3$ is a projection in general position of $K$ in $\mathbb{R}^4$ to a subspace $\mathbb{R}^3$.

Two diagrams represent the same knot if and only if they can be related by a finite sequence of {\em Roseman moves} which are shown on Fig. \ref{roseman} (this figure is due to K. Kawamura, K. Oshiro and K. Tanaka, see, for example, \cite{Tan}).

\begin{figure}
\centering\includegraphics[width=450pt]{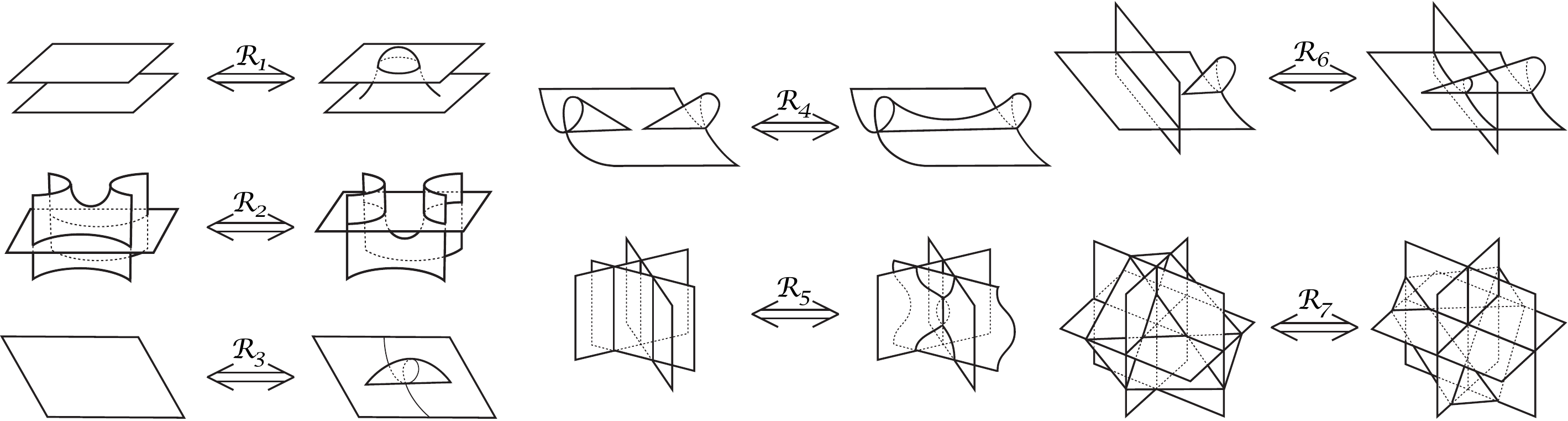}
\caption{Roseman Moves}
\label{roseman}
\end{figure}

Another way of representing 2-knots is via {\em spherical diagrams}. Those diagrams are a 2-dimensional analog to chord diagrams of 1-knots.

\begin{dfn}
\label{def:gauss}
A {\rm spherical diagram} (or Gauss diagram) is a 2-complex consisting of a sphere $S$ and a set $D$ of marked curves on it such that:
\begin{enumerate}
\item Every curve is either closed or ends with a cusp, the number of cusps is finite;
\item Every curve of the set $D$ is paired with exactly one curve of that set, one of the paired curves is marked as {\rm upper}, both curves are oriented (marked with arrows);
\item Two curves ending in the same cusp are paired and both arrows either look towards the cusp or away from it;
\item If two curves intersect, the curves paired with them intersect as well (thus a triple point appears on the sphere $S$ three times, see the upper part of Fig. \ref{second_smoothing});
\item All triple points are geometric.
\end{enumerate}
\end{dfn}

Roseman moves can be naturally translated to the language of spherical diagrams. If one forgets the ``over -- under'' information on a spherical diagram and considers such diagrams modulo Roseman moves, one obtains {\em free 2-knots}.

In 1-dimensional knot theory {\it smoothing} of a crossing means cutting out a crossing with its small neighbourhood and reglueing two pairs of half-edges in one of the possible two ways (see Fig. \ref{1-smooth}).

\begin{figure}
\centering\includegraphics[width=50pt]{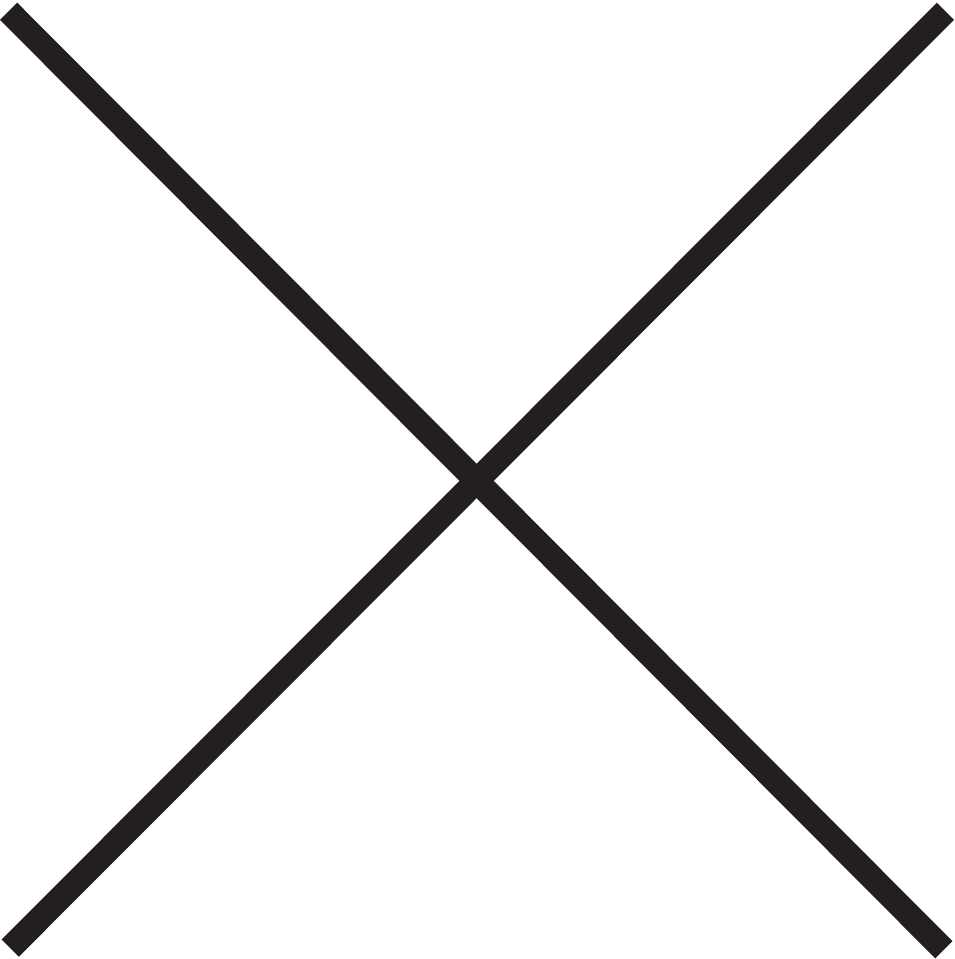} \\
\vspace{10mm}
\centering\includegraphics[width=50pt]{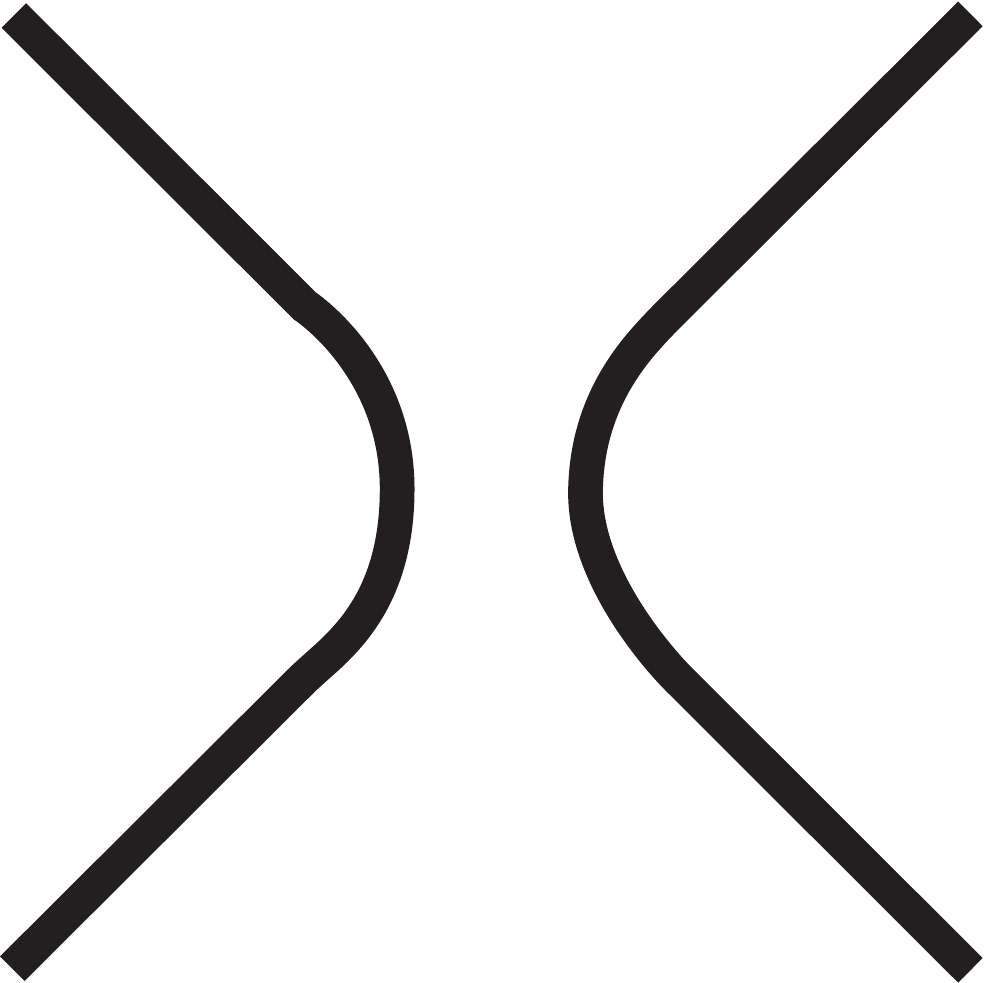} ~~~~~~~ \includegraphics[width=50pt]{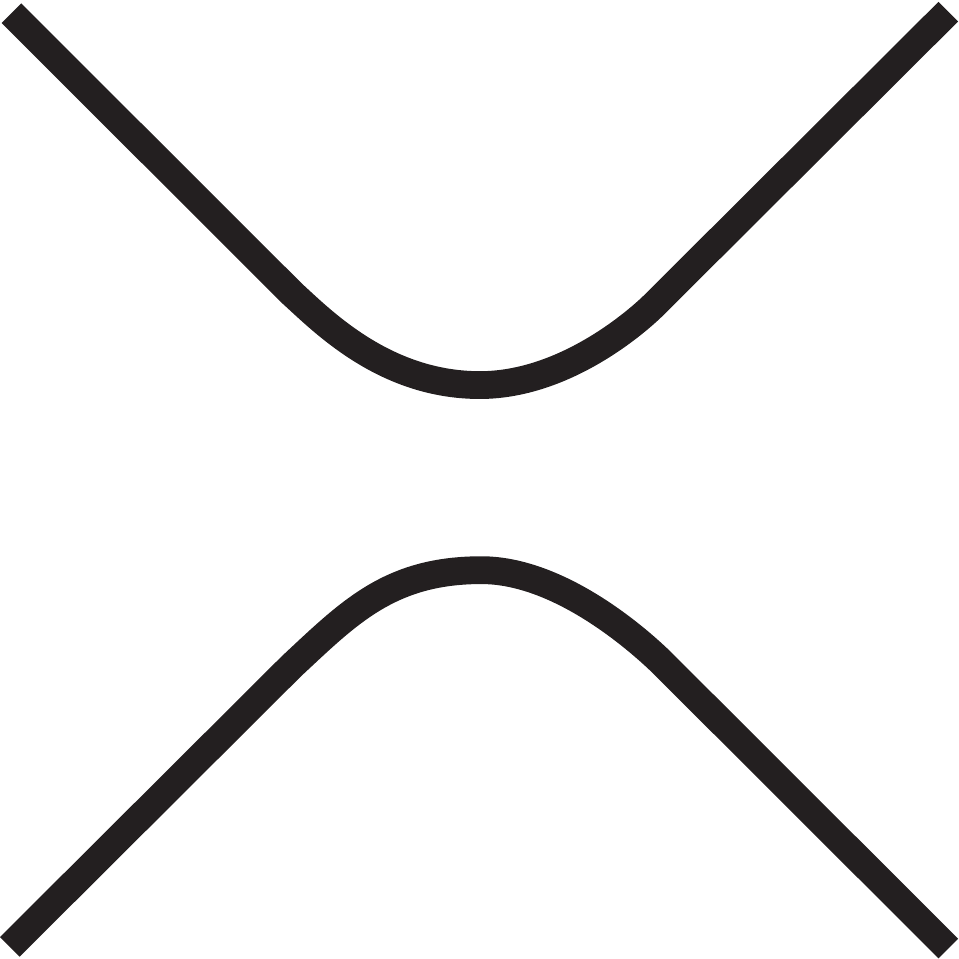}
\caption{Smoothings of a crossing}
\label{1-smooth}
\end{figure}

In 2-dimensional case not crossings, but double lines are smoothed. That is, locally one can think of a 2-dimensional smoothing as a 1-dimensional smoothing on a transversal section, multiplied by a segment. The trick is to define smoothing of a whole double line (or a family of such) in a compatible way. Here we present one approach, using spherical diagrams. 

Consider a pair of paired curves $\gamma, \gamma'$ on a spherical diagram of a knot. Cut the diagram along those lines. The resulting multi-component complex has four boundary components: $\gamma_1, \gamma_1'$ on the sphere, and $\gamma_2, \gamma_2'$ on the cut out pieces.

Now we glue those curves back together (respecting their orientation and intersecting curves) following the rule: non-prime curve must be glued to a prime one, see Fig. \ref{first_smoothing}.

\begin{figure}
\centering\includegraphics[width=300pt]{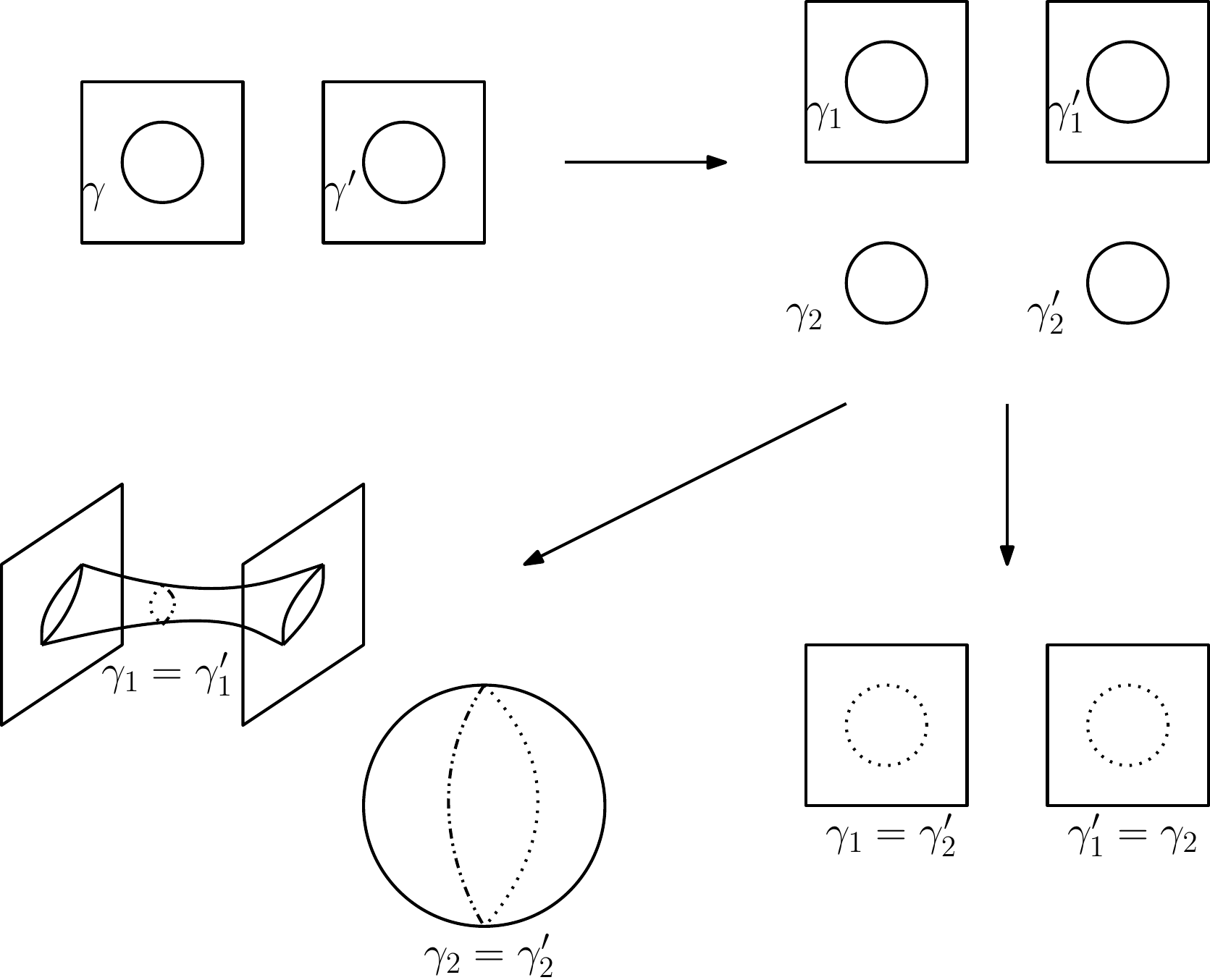} \\
\caption{Smoothings of a double line: the curve $\gamma$ is paired with $\gamma'$; the spherical diagram is cut along those lines and reglued in two possible ways}
\label{first_smoothing}
\end{figure}

To complete our smoothing procedure we need to understand, what happens to the triple points on the smoothed line.

It is easy to see that the reglueing of $\gamma$'s naturally induces a (1-dimensional) smoothing near the third preimage of a triple point involved, see Fig. \ref{second_smoothing}. The type of this smoothing exactly corresponds to the choice of pairing between the boundary components $\gamma$: $\gamma_1$ to $\gamma_1'$ or to $\gamma_2'$. 

\begin{figure}
\centering\includegraphics[width=300pt]{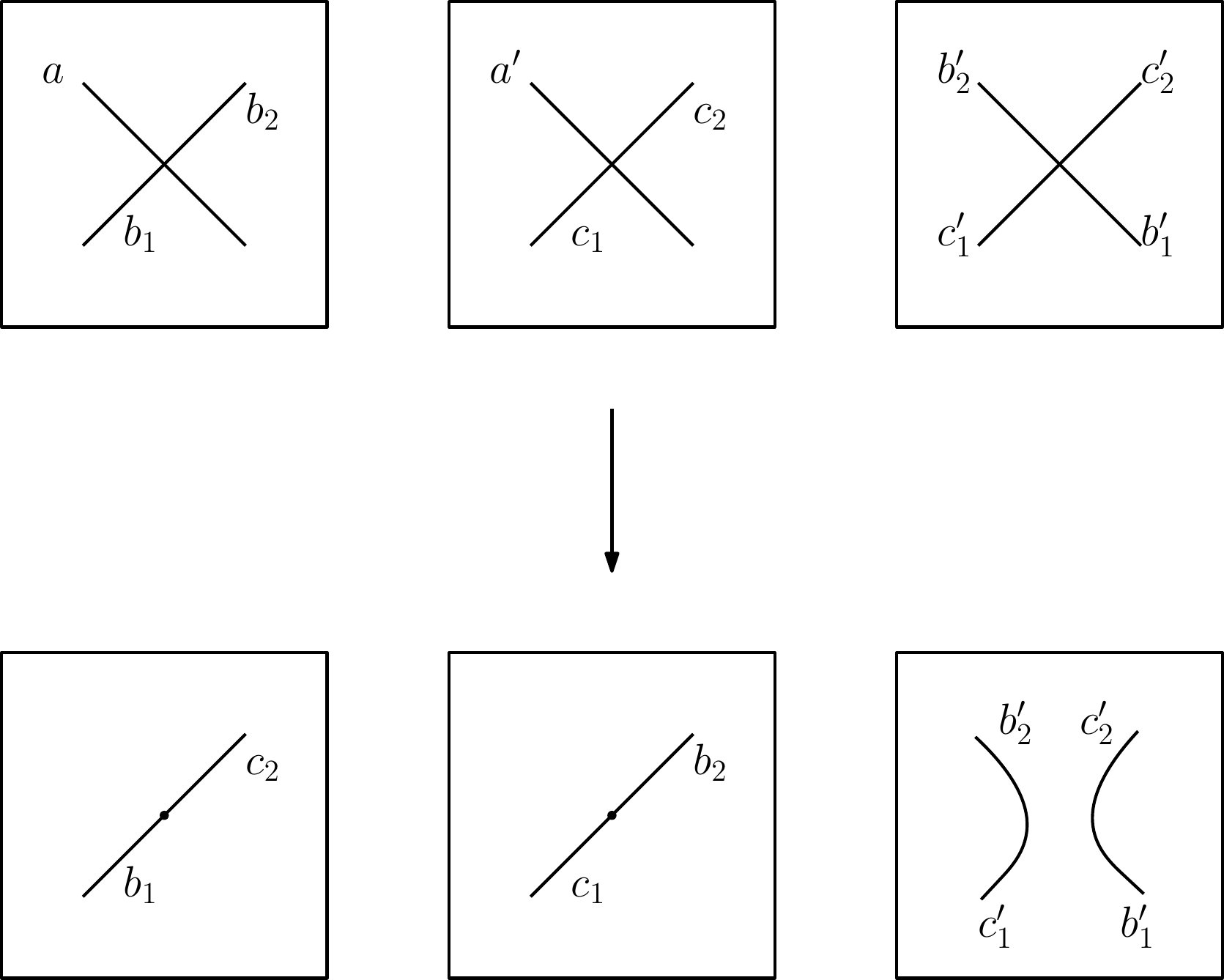} \\
\caption{Smoothing of the double line $a$ removes $a$ and $a'$ from the diagram and induces a 1-dimensional smoothing of the crossing of the curves $b'$ and $c'$}
\label{second_smoothing}
\end{figure}

Note, that the smoothing of a self-intersecting double line is defined correctly. In that case two cusps are produced, see Fig. \ref{third_smoothing}.

\begin{figure}
\centering\includegraphics[width=300pt]{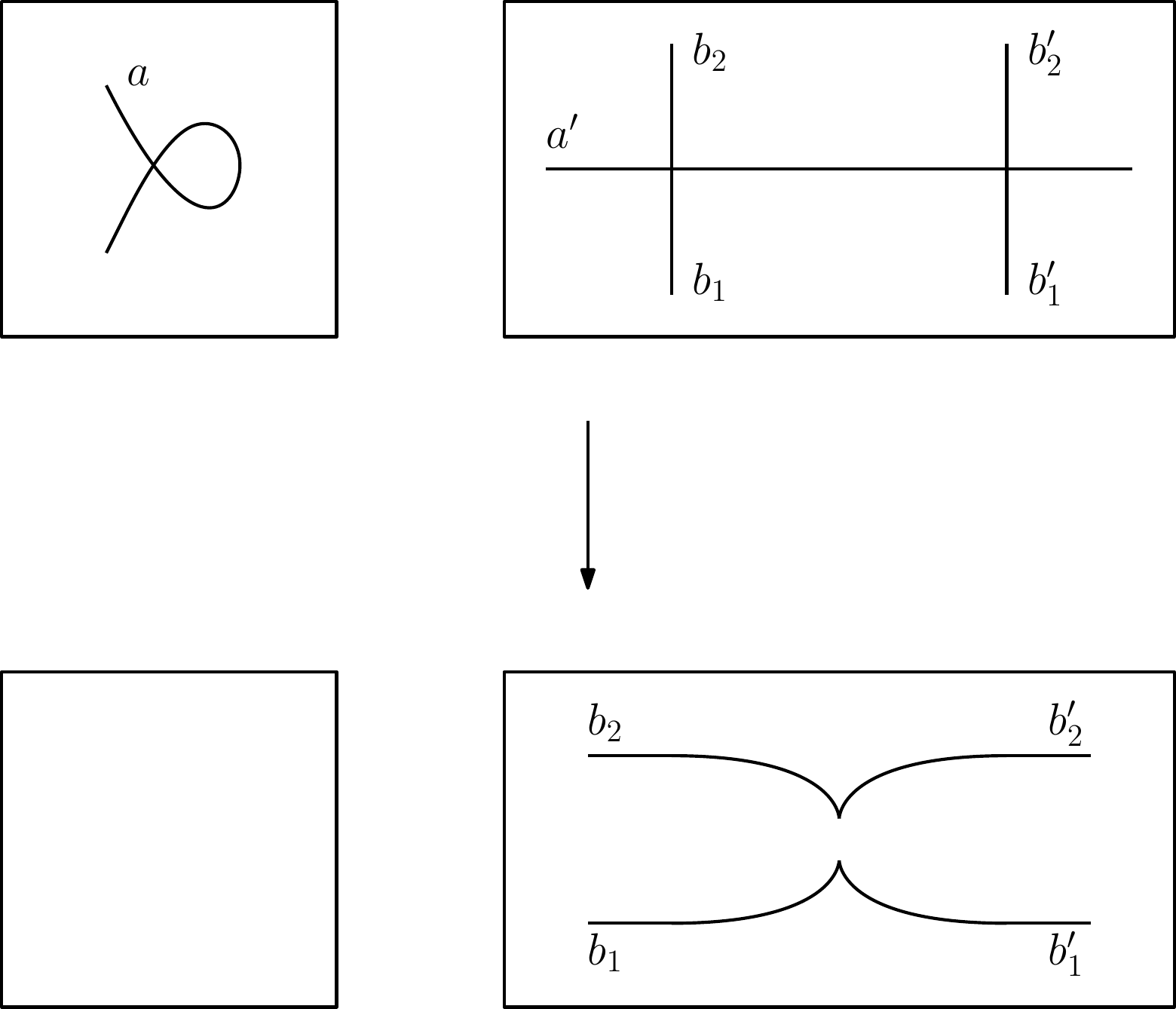} \\
\caption{Smoothing of a self--crossing double line $a$ removes $a$ and $a'$ from the diagram and produces cusps}
\label{third_smoothing}
\end{figure}

That procedure defines a {\it smoothing} of a double line. Smoothing of a set of double lines is done one by one in arbitrary order.

The following lemma is needed to prove the main theorem \ref{mnthm}:

\begin{lm}
\label{smoothing_lemma}
For every free 2-knot diagram there exists a smoothing yielding the trivial knot.
\end{lm}

\begin{proof}
Consider a smoothing of the full set of double lines of a given diagram.

Every double line can be smoothened in two ways. One of them produces an additional connected component (thus transforming the knot into a link) and the other one doesn't change the number of connected components.

Let us choose the smoothing of every double line not changing the number of connected components. Thus, after smoothing of every double line we obtain a new diagram of a free 2-knot. Note, that every smoothing either reduces the number of double lines of the diagram or (in case of a self--intersecting double line) doesn't change it but reduces the number of triple points. Thus the total number of triple points and double lines strictly decreases with every smoothing.

That means that the process of smoothing of all the double lines of a diagram is finite and yields an empty spherical digram, in other words --- a diagram of the trivial knot. 
\end{proof}

Now we are ready to prove Theorem \ref{mnthm}.

In our case the parity is defined as in \cite{Cobordism, Parity}: for every point on the complex with exactly two preimages on the disc we connect those preimages by a curve $\gamma$ transversally intersecting the set of double lines in a finite number of points, not going through triple points and behaving in a compatible way near the endpoints. Then we count the number of intersection between $\gamma$ and the double lines mod 2. It is easy to verify that this number is constant along a double line and all the parity axioms are satisfied (see \cite{Parity}), namely, a double line ending in a cusp is even and among the tree double lines meeting in a triple point either zero or two of them are odd. This parity is called {\em Gaussian}. \\

Now consider the graph $\Gamma$ and the spanning complex $K$. The complex $K$ has double lines of two types: double lines with their endpoints on the boundary and the ``interior'' double lines. The smoothing operation defined above can be easily generalized to the case of interior double lines on a complex with boundary.

Let us smooth all interior double lines of the complex $K$. Due to the smoothing lemma \ref{smoothing_lemma} we thus obtain a 2-complex $\tilde{K}$ which is an image of a disc with no interior double lines, that is all double lines of the complex $\tilde{K}$ have their endpoints on the boundary of the complex. We claim that this complex has neither cusps no triple points.

Indeed, the smoothing operation of the interior double lines leaves the neighbourhood of the boundary $\partial K$ intact. The Gaussian parity of the double lines with their ends on the boundary is the same as the Gaussian parity of the corresponding vertices of the graph $\Gamma$ (that is, the crossings of the free knot). That means that all the double lines of the complex $\tilde{K}$ are odd. But due to parity properties, every double line ending with a cusp is even and there is at least one even double line among the three double lines meeting in a triple point. Thus, the complex $\tilde{K}$ has no cusps and triple points.

The double lines of the complex $\tilde{K}$ define a pairing with no intersections of the chords of the diagram corresponding to the graph $\Gamma$. Thus, Theorem \ref{mnthm} is proved. \\

The authors are grateful to Scott Carter and Seiichi Kamada for various useful discussions.


\end{document}